\documentclass[letterpaper, 10pt]{article}
\usepackage[letterpaper]{geometry}
\usepackage{amsmath, amssymb, amsthm}
\usepackage{graphicx}

\newtheorem{theorem}{Theorem}
\newtheorem{cor}{Corollary}[theorem]
\newtheorem{lemma}[theorem]{Lemma}
\newtheorem{prop}[theorem]{Proposition}

\newcommand{\R}{\mathbb R}
\newcommand{\Q}{\mathbb Q}
\newcommand{\Z}{\mathbb Z}
\newcommand{\N}{\mathbb N}
\newcommand{\C}{\mathbb C}

\newcommand{\SL}{\operatorname{SL}}

\renewcommand{\phi}{\varphi}

\renewcommand{\epsilon}{\varepsilon}
\renewcommand{\H}{\mathcal{H}}
\renewcommand{\Im}{\operatorname{Im}}
\renewcommand{\Re}{\operatorname{Re}}
\newcommand{\Area}{\operatorname{Area}}
\newcommand{\del}{\partial}
\newcommand{\X}{M}

\title{Counting elliptic curves of bounded Faltings height}
\author{Ruthi Hortsch}
\date{\today}

\begin{document}

\maketitle

\maketitle

\renewcommand{\thefootnote}{}



\renewcommand{\thefootnote}{\arabic{footnote}}
\setcounter{footnote}{0}


\begin{abstract}
We give an asymptotic formula for the number of elliptic curves over $\Q$ with bounded Faltings height. Silverman \cite{Sil} has shown that the Faltings height for elliptic curves over number fields can be expressed in terms of modular functions and the minimal discriminant of the elliptic curve. We use this to recast the problem as one of counting lattice points in a particular region in $\R^2$.
\end{abstract}

\section{Introduction}

Let $\mathcal{E}$ be the set of isomorphism classes of elliptic curves over $\Q$. There are a number of invariants of elliptic curves we have yet to understand well (in particular, the rank of the rational points of an elliptic curve, or the size of its $n$-Selmer group). To discuss what value these take ``on average", we need to define a measure on $\mathcal{E}$. This is often done by defining a height, a function $H: \mathcal{E} \rightarrow \R$ such that if $\mathcal{E}_{H<X} = \{ E \in \mathcal{E} \mid H(E)<X\}$, then $\# \mathcal{E}_{H<X}$ is finite. We can then make sense of what an average value is on a finite set, and hope a limit exists as $X \rightarrow \infty$. To use this to measure the invariants, we first need to understand just how quickly $\# \mathcal{E}_{H<X}$ grows.

Commonly used is the {\em naive height} $H_N$: Every elliptic curve over $\Q$ can be written uniquely as $y^2=x^3+Ax+B$ where $A,B \in \Z$ are such that there is no prime $p$ with $p^4 \mid A$ and $p^6 \mid B$. If we call this curve $E_{A,B}$, then its naive height is
\[ H_N(E) = \max(B^2, |A^3|).\]
Using sieve methods or M\"obius inversion (see \cite{Bru}), one can show that  \[ \# \mathcal{E}_{H_N<X} = 4\zeta(10)^{-1}X^{5/6} + O(X^{1/2}) \]
(for related calculations over number fields, see \cite{Bek}). This formula comes from showing that the number of lattice points of a region is roughly its area with an error of its perimeter. The region where naive height is less than $X$ is the rectangle of length $2X^{1/3}$ and height $2X^{1/2}$, which has exactly area $4X^{5/6}$ and perimeter $O(X^{1/2})$ (the $\zeta(10)$ comes from excluding lattice points corresponding to nonminimal models).

Counting elliptic curves using other heights is more difficult. For example, counting elliptic curves of bounded discriminant is difficult because the region of points with bounded discriminant has cusps, by which is meant there are points with large $A,B$ and small discriminant. Controlling these cusps is difficult, even if assuming the ABC conjecture.

Brumer and McGuinness \cite{BruMcG} have given a heuristic for the number of elliptic curves with positive (respectively negative) discriminant up to a bound, and Watkins \cite{Wat} has used this to give heuristics for the average rank counted this way, as well as heuristics for elliptic curves of bounded conductor. It is generally believed that the average rank should be the same for each of these heights. However, no proof has been given of these conjectures.

In this paper, we show that if $h_F$ is the Faltings height, then \[ \# \mathcal{E}_{h_F<Y} = 12 \sigma \zeta(10)^{-1} e^{10Y} + O(e^{6Y}Y^3)\] where $\sigma$ is an absolute constant that we express as a specific integral given Section~\ref{faltht}. Note that if we rewrite this in terms of $H_F = e^{12h_F}$, this looks similar to the equation for the naive height: \[ \# \mathcal{E}_{H_F<X} = 12 \sigma \zeta(10)^{-1} X^{5/6} + O(X^{1/2}(\log X)^3).\] As in the naive height case, this comes from approximating lattice points in a region by the area with an error from the perimeter: $\sigma X^{5/6}$ is the area, while $12 \zeta(10)^{-1}$ adjusts it to consider only curves up to isomorphism (which is controlled by only considering certain residue classes of integral points). This is more difficult than the naive height case because, like the discriminant, the relevant region has a cusp with unbounded points, but these cusps can be controlled more easily than in the discriminant case.

\section{Faltings Height}\label{faltht}

In his proof of the Mordell Conjecture, Faltings introduces a height on the set of abelian varieties over a number field $K$, referred to as the {\em Faltings height} (see \cite{CorSil} for a longer overview, or the original paper \cite{Fal}). Silverman has given a formulation for elliptic curves, which gives the Faltings height as logarithmic (in \cite{Sil}). For the sake of simplicity we will give this as a definition.

If $h_F(E)$ denotes the original Faltings height for an elliptic curve $E$ defined over a number field $K$, let $H_F(E) = e^{12h_F(E)}$. If $v$ is an infinite place of $K$, define 
\begin{align*}
\epsilon_v &= 1, 2 \text{ if $v$ is real, complex}\\
\Delta_E^{min}& \text{ the minimal discriminant}\\
\tau_v \in \H &\text{ (the complex upper half plane) such that }E(\overline{K_v}) \cong \C/(\Z+\tau_v\Z)\\
\Delta(\tau)& = (2\pi)^{12}q_\tau \prod_{n=1}^\infty (1-q_\tau^n)^{24} \text{ where }q_\tau = e^{2\pi i\tau}.
\end{align*}
Using Silverman's reformulation, we define
\[H_F(E) = \frac{|N_{K/\Q}\Delta_E^{min}|}{\prod_{v \mid \infty}|\Delta(\tau_v)|^{\epsilon_v}(\Im\tau_v)^{6\epsilon_v}}\]
which simplifies to
\[ \frac{\Delta_E^{min}}{|\Delta(\tau)| (\Im \tau)^6}\]in the case where $K=\Q$.
Since $\Delta(\tau)$ is a modular form of weight 12, it is easy to check that this is independent of the choice of $\tau$. (Silverman\cite{Sil}, Faltings\cite{Fal}, and Deligne\cite{Del} all use different normalizations of the Faltings height; this normalization agrees with that of Faltings.) Also see \cite{Loeb} for a good background and summary.

Silverman uses this to show that for all $\epsilon >0$, there exist $C_1, C_2(\epsilon)$ such that for all elliptic curves $E$ over $\Q$ \[ C_1 H_F(E) \le H_N(E) \le C_2(\epsilon)H_F(E)^{(1+\epsilon)}. \]

This inequality tells us that the Faltings height is very similar to the naive height, but not that it is within a bounded factor of the naive height. So counting elliptic curves by naive height does not give good bounds for counting curves by Faltings height.

We will study the number of isomorphism classes of elliptic curves $E$ with $H_F(E)<X$ for large $X$. Say that $(A,B)$ are {\em weakly minimal with respect to a prime $p$} provided that either $p^4 \nmid A$ or $p^6 \nmid B$; if this holds true for all primes, simply say $(A,B)$ is {\em weakly minimal}. Let $E_{A,B}$ indicate the elliptic curve corresponding to $y^2=x^3+Ax+B$, and \[S_X = \{ \text{weakly minimal } (A,B) \in \Z^2 \mid 4A^3+27B^2 \not = 0, \; H_F(E_{A,B})<X\}.\] The elliptic curves $E_{A,B}$ such that $(A,B) \in S_X$ are representatives of the isomorphism classes of elliptic curves with Faltings height less than $X$. 

For $\tau$ in the upper complex plane $\H$, let $j(\tau)$ denote its $j$-invariant, and $\Delta(\tau)$ the modular discriminant as defined above. By traditional fundamental domain (with respect to the $j$-invariant), we mean $|\tau|\ge 1$ and $-\frac{1}{2}< \Re(\tau)\le \frac{1}{2}$. In this paper we prove the following:

\begin{theorem}\label{maintheorem}
For $t \in \R$, let $\tau_t \in \H$ such that $j(\tau_t) = 6912t/(4t+27)$ and $\tau_t$ is in the traditional fundamental domain. Let
\[\sigma = \frac{2}{5} \int_{-\infty}^{\infty} t^{-2/3} \left| \frac{\Delta(\tau_t)\Im(\tau_t)^6}{16(4t+27)}\right|^{5/6}dt.\] Then
\[ \# S_X = 12 \sigma \zeta(10)^{-1} X^{5/6} + O(X^{1/2}(\log X)^3).\]
\end{theorem}

Roughly, the intuition for where these numbers come from is as follows: $\sigma X^{5/6}$ is the area of the two dimensional region in $\R^2$ consisting of $(A,B)$ such that the corresponding elliptic curve $E_{A,B}$ has Faltings height less than $X$, $12\zeta(10)^{-1}$ corrects for the fact that we want only to take weakly minimal $(A,B)$, and the error comes from the fact that estimating lattice point counts by area will have an error related to the boundary (in particular, the boundary of a bounded version of this region). When the integral defining $\sigma$ is evaluated, we get approximately $\sigma \approx 29089$, which means the constant of the leading term is approximately $348716$.

This result is more challenging to prove than one might expect since $H_F$ is not simply bounded below by a constant times $H_N$, so we cannot simply apply results known about the naive height. Additionally, the calculation involves counting lattice points in an unbounded region whose boundary is given by a \emph{transcendental} equation, which rules out many standard approaches.

\section{Defining the region of interest}

Assume that we have $A,B$ such that $4A^3+27B^2 \not = 0$. For $A,B \in \Q$, let $\Delta_{E_{A,B}}^{min}$ denote the minimal discriminant (note that over $\Q$ there is a global minimal model, and we could take $\Delta_{E_{A,B}}^{min}$ to just be the polynomial discriminant of this minimal model). Let 
\begin{align*}
\tau_{A,B} &\in \H \text{ be such that } |\tau_{A,B}|\ge 1, \textstyle -\frac{1}{2}<\Re\tau \le \frac{1}{2},\\ & \text{ and } E_{A,B}(\C) \cong \C/(\Z + \tau_{A,B} \Z)\\
\Delta_{A,B} & = -16(4A^3+27B^2)\\
j_{A,B} &= -1728\frac{(4A)^3}{\Delta_{A,B}}\\
\lambda_{A,B} &= \frac{|\Delta_{E_{A,B}}^{min}|}{|\Delta_{A,B}|}.
\end{align*}
Note that the first three of these make sense for $A,B \in \R$, while the last requires $A,B \in \Q$. Furthermore, $\lambda_{A,B}$ will be one of four values depending on whether the model given for $E_{A,B}$ is minimal at 2 and 3 (we will discuss this in section \ref{resclasses}). To clarify things later, note we have three values in this paper which have a delta in their notation denoting slightly different things: $\Delta_E^{min}$ (the minimal discriminant), $\Delta_{A,B}$ (the polynomial discriminant), and $\Delta(\tau)$ (the modular discriminant).

Silverman's theorem tells us that
\begin{align*}
H_F(E_{A,B}) &= \frac{|\Delta_{E_{A,B}}^{min}|}{|\Delta(\tau)|\Im(\tau)^6} \\
&= \lambda_{A,B}\frac{|\Delta_{A,B}|}{|\Delta(\tau_{A,B})|\Im(\tau_{A,B})^6}
\end{align*}
This motivates us to define the function \[f(A,B) = \left | \frac{\Delta(\tau_{A,B}) \Im(\tau_{A,B})^6}{\Delta_{A,B}}\right |^{1/2}\] which is well-defined for all $A,B \in \R$ where $\Delta_{A,B} \not = 0$ (we invert and take a square root to make later calculations easier). Fixing a $\lambda>0$, we would like to know how many weakly minimal integer points there are in
\[\{ (A,B) \in \R^2 \mid \lambda f(A,B)^{-2}< X \}.\] 

To ease this, we define
\[R_{X,\lambda} =  \{ (A,B) \in \R^2 \mid \lambda f(A,B)^{-2}< X \} \cup\{ (A,B) \in \R^2 \mid \Delta_{A,B} = 0\}.\]
Note that since $R_{X,\lambda} = R_{X/\lambda, 1}$, we may as well just study $R_X = R_{X,1}$. Furthermore, $f$ is a weighted homogeneous function: replacing $(A,B)$ with $(X^{1/3}A, X^{1/2}B)$ will scale $f(A,B)$ by $X^{-1/2}$. So $R_X$ is $R_1$ scaled by $X^{1/3}$ in the $A$-direction and $X^{1/2}$ in the $B$-direction.

The next section calculates the area of $R_1$ (the value denoted $\sigma$), while the rest of the paper is focused on setting up the proof of the main result. We will show that for large but finite $A,B$, there are no more relevant lattice points in $R_X$. Scaling down to $R_1$, we can show that $R_X$ has a finite boundary if contained in a finite box, but we want to measure the actual length of this boundary. We do this by showing that in a small enough box the boundary length does not exceed the box length, which follows from showing (asymptotic) uniformity in the cusps.

\section{Area of $R_1$}

To better calculate the area, we do a change of variables to the parameter $\displaystyle t=\frac{A^3}{B^2}$. Then $\displaystyle j_{A,B} = \frac{6912t}{4t+27}$ and we can also define a function \[f(t) = \left | \frac{\Delta(\tau_t) \Im(\tau_t)^6}{16(4t+27)}\right |^{1/2}\] where $\tau_t$ is in the fundamental domain such that $j(\tau_t) = 6912t/(4t+27)$. Then $R_1$ is by definition the area where \[ \frac{B^2}{f(t)^2}<1\] or alternatively, where $-f(t)<B<f(t)$. So with a change of variables we get that
\begin{align*}
\Area(R_1) & = \iint_{R_1} dAdB \\
& = \int_{-\infty}^{\infty} \int_{-f(t)}^{f(t)} \left(\frac{1}{3}t^{-2/3}B^{2/3}\right)dBdt\\
&= \frac{1}{3} \int_{-\infty}^{\infty} t^{-2/3} \left( \frac{3}{5} B^{5/3}\right)\bigg|_{-f(t)}^{f(t)} dt\\
&= \frac{2}{5} \int_{-\infty}^{\infty} t^{-2/3} f(t)^{5/3}dt
\end{align*}

Using a computer to evaluate, we conclude $\Area(R_1)\approx 29089$. We will denote $\sigma = \Area(R_1)$, and note that $\Area(R_X) = \sigma X^{5/6}$. This is the same $\sigma$ that appears in Theorem~\ref{maintheorem}. Note that in the cusp, near $\Delta_{A,B}=0$, we have that $t \approx -27/4$, and both $\tau_t$ and $f$ grow large.

(If we wish to have a somewhat cleaner integral we can also rewrite as follows: Use $j(z) = 6912t/(4t+27)$ to do a change of variables to a complex number $z$ over an appropriate curve. Rearranging and applying that $\Delta(z) = (2\pi)^6 \frac{j'(z)^6}{j(z)^4(j(z)-1728)^3}$, the integral is
\[ \int_{\gamma} |\Delta(z)| \Im(z)^5 \frac{j'(z)}{|j'(z)|} dz \]
over a curve $\gamma$ in the upper complex plane of all $z$ with real $j(z)$.)


\section{The rough shape of $R_X$}

\begin{lemma}\label{DiscBound}
The function $|\Delta(\tau)\Im(\tau)^6|$ is bounded on the complex upper half-plane $\H$.
\end{lemma}

\begin{proof}
Since $\Delta(\tau)\Im(\tau)^6$ is invariant under $\SL_2(\Z)$, assume that $\tau$ is in the standard fundamental domain, so that $|\tau| \ge 1$ and $|\Re(\tau)|\le 1/2$. If we can show that $\Delta(\tau)\Im(\tau)^6$ is finite as $\Im(\tau) \rightarrow \infty$, then for any $\tau$ with $\Im(\tau)$ above some fixed constant, it is bounded. The section of the fundamental domain that has $\Im(\tau)\le$ that constant is a compact domain, and thus the function must also be bounded on that domain.

It remains only to show that the function is finite as $\Im(\tau) \rightarrow \infty$. By definition \[ \Delta(\tau) = \frac{1}{(2\pi)^{12}} q \prod_{n=1}^\infty (1-q^n)^{24}\] where $q=e^{2\pi i \tau}$. So as $\Im(\tau) \rightarrow \infty$, $q\rightarrow 0$, and $|\Delta(\tau)| = O(q) = O(e^{-2\pi \Im(\tau)})$. Thus, as $\Im(\tau)$ grows large, $|\Delta(\tau)|\Im(\tau)^6$ tends to zero.
\end{proof}

\begin{cor}\label{DenBound}
 Let $C$ be a positive constant such that $|\Delta(\tau)|\Im(\tau)^6 < C$ for all $\tau \in \H$. If $(A,B) \in R_X$ then $|\Delta_{A,B}|<CX$.\end{cor}

The above follows immediately from Lemma~\ref{DiscBound}. The following lemma will show that points in $R_X$ are either small compared to $X$ or close to the curve given by $\Delta_{A,B}=0$.
\begin{figure}
\centering
\includegraphics[height=3in]{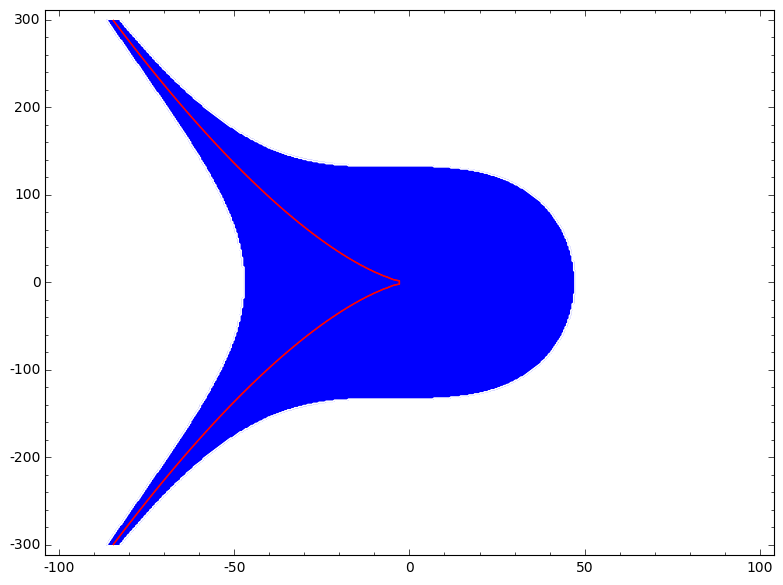}
\caption{The region $R_1$ in blue, with the cubic $4A^3+27B^2$ in red}
\end{figure}

\begin{lemma}\label{InTail}
Let $(A,B) \in R_X$. If $|B|^2<CX/27$, then $|A|^3<CX/2$. Otherwise, $A = -cB^{2/3}+\epsilon X B^{-4/3}$ for some $|\epsilon|< C$ and $c= (27/4)^{1/3}$.
\end{lemma}

\begin{proof}
Suppose $|B|^2<CX/27$ and $|A|^3\ge CX/2$, then
\begin{align*}
|\Delta_{A,B}| & =16|4A^3+27B^2|\\
& \ge 16(4|A^3| - 27|B|^2) \\
& \ge 16(2CX - CX)\\
&= 16CX \ge CX
\end{align*}
 and thus $(A,B) \not \in R_X$.

Let $B\not = 0$ and $A=-cB^{2/3} + \epsilon XB^{-4/3}$ with $|\epsilon| \ge C$. Then \[ |\Delta_{A,B}| = 64|\epsilon| X (3c^2-3c\epsilon X B^{-2} + \epsilon^2 X^2 B^{-4}).\] Well the polynomial $x^2-3cx+3c^2\ge 1$ for all $x \in \R$ and by assumptions $|\epsilon|>C$, so $|\Delta_{A,B}|>64CX>CX$. This implies $(A,B) \not \in R_X$.
\end{proof}

Note that this tells us only in a very rough sense that points in $R_X$ are close to the cubic given by $\Delta_{A,B}=0$. In the next section, we will give a stronger version of this statement (which holds less generally), and use that to count the lattice points in the region $R_X$. It is in that calculation differs from considering only elliptic curves of bounded discriminant.

\section{Bounding the size of lattice points}

In this section we will show that all large enough integral $(A,B) \in R_X$ have $\Delta_{A,B}=0$. This is done in two steps: first, we quantify how close a point has to be to curve $\Delta_{A,B}=0$ to have the property $|\Delta_{A,B}|<1$. Then we show that, for large enough $A$, a point farther from the curve cannot be in $R_X$.

\begin{lemma}\label{discsmall}
Assume $B \not =0$. Let $\epsilon_0$ be the positive real root of $64x(3c^2+x^2)-3/4$, recalling that $c= (27/4)^{1/3}$ (this root is approximately 0.0011). For all $|\epsilon|<\epsilon_0$, if $A=-cB^{2/3} + \epsilon B^{-4/3}$ then $|\Delta_{A,B}| <1$.
\end{lemma}
\begin{proof}
Under the assumptions, we have that 
\begin{align*}
|\Delta_{A,B}| & = 64 |\epsilon| |3c^2 - 3c\epsilon B^{-2} + \epsilon^2 B^{-4}|\\
& \le 64 |\epsilon| (3c^2 + 3c|\epsilon| B^{-2} + \epsilon^2 B^{-4})\\
&< 64 \epsilon_0 (3c^2 + \epsilon_0^2) + 192\epsilon_0^2c\\
&= 3/4 + 192\epsilon_0^2c\\
&<1
\end{align*}
since $192\epsilon_0^2c$ is much smaller than $1/4$.
\end{proof}

The following lemma is a strong version of Lemma~\ref{InTail} that only holds for large enough $X$ and $A$.

\begin{lemma}\label{heightbig}
There are positive constants $\X$, $N$ such that for all $X\ge \X$, if $|A|>N X^{1/3}(\log X)^2$ and $A=-cB^{2/3}+\epsilon B^{-4/3}$ where $|\epsilon|\ge \epsilon_0$, then $(A,B) \not \in R_X$.
\end{lemma}

\begin{proof}
Assume that we have taken $\tau_{A,B}$ in the usual fundamental domain, so that $|\Re(\tau_{A,B})|\le1/2$ and $|\tau_{A,B}|\ge 1$. We note that
\begin{align*}
f(A,B)^{-2} & = \frac{|\Delta_{A,B}|}{|\Delta(\tau_{A,B})| \Im(\tau_{A,B})^6}\\
&= \frac{1728|4A|^3}{|\Delta(\tau_{A,B})| |j(\tau_{A,B})|\Im(\tau_{A,B})^6}
\end{align*}

However, since $|\Delta(\tau_{A,B})| |j(\tau_{A,B})| \rightarrow 1$ as $\tau \rightarrow i\infty$, there is some positive constant bounding $|\Delta(\tau_{A,B})||j|$ from above, so \begin{align*}
f(A,B)^{-2} \ge c_1 \frac{ |A|^3}{\Im(\tau_{A,B})^6}
\end{align*}
for some constant $c_1>0$.

If $\Im(\tau) \le 1$, this means $f(A,B)^{-2} \ge c_1 |A|^3 > c_1 \N^3 X (\log X)^6$. So as long as $\X$ is large enough so $c_1 \N^3 (\log \X)^6 >1$, this tells us that $f(A,B)^{-2} \ge X$ and thus $(A,B) \not \in R_X$. (We will need to be careful to make sure this is compatible with how we choose $N$.)

Assume $\Im(\tau) >1$. Note that \[|\Delta_{A,B}| = 64 | \epsilon| |3c^2-3c\epsilon B^{-2} + \epsilon^2 B^{-4}|.\] The polynomial $x^2-3cx+3c^2$ is positive for all real numbers, with a minimal value of $f(3c/2) \ge 2$, which implies that $|\Delta_{A,B}| \ge 128\epsilon_0 > 0.1$. Since $j = 1728(4A)^3\Delta_{A,B}^{-1}$, this means that $|j| \le c_2 |A|^3$ for some positive constant $c_2$.

But since $|q_\tau| = e^{-2\pi\Im(\tau)}$, and for $\Im(\tau)>1$, $|\log(|j|) - \log(|q_\tau^{-1}|)|$ is bounded, it follows that \[ \Im(\tau) < c_3 \log (|A|) + c_4\] for some positive constants $c_3, c_4$.

So \begin{align*}
f(A,B)^{-2} & \ge c_1 \frac{|A|^3}{\Im(\tau_{A,B})^6} \\
&\ge c_1 \frac{|A|^3}{(c_3\log(|A|)+c_4)^6}
\end{align*}

If we take $N$ such that $3c_1N^3 > 2 c_3$, then there is some constant $c_5$ such that if $X>c_5$ and $|A|>N X^{1/3} (\log X)^2$,  \[c_1 \frac{|A|^3}{(c_3\log(|A|)+c_4)^6} > X.\] Which implies that $(A,B) \not \in R_X$.

So we can take any $N$ such that $3c_1N^3 > 2c_3$ and $M$ such that $c_1N^3 (\log M)^6 >1$ and $M>c_5$, and the lemma holds for these.
\end{proof}

Together Lemmas \ref{discsmall} and \ref{heightbig} tell us that for $X\ge \X$, if $|A|>N X^{1/3} (\log X)^2$, then either $|\Delta_{A,B}|<1$ or $(A,B)\not \in R_X$. Thus, the only integer points in $R_X$ where $|A|>N X^{1/3} (\log X)^2$ are those where $\Delta_{A,B}=0$.

\section{Weakly minimal curves not minimal at $2$ or $3$}\label{resclasses}

Ultimately what we want to count is \[\#\{ (A,B) \in \Z^2 \mid (A,B) \text{ are weakly minimal}, H_F(E_{A,B})<X \}.\] We defined $f$ so that $H_F(E_{A,B})<X$ is equivalent to $(A,B) \in R_{X,\lambda_{A,B}} = R_{X/\lambda_{A,B}}$, so to use our study of $R_X$, we need to more carefully examine $\lambda_{A,B}$.

Recall that we defined $\lambda_{A,B} = \lambda_{E_{A,B}}$ so that it relates the minimal and polynomial discriminants so that $|\Delta_{E_{A,B}}^{min}|=\lambda_{A,B}|\Delta_{A,B}|$. Since we are requiring that $(A,B)$ be weakly minimal, it follows that the model $E_{A,B}$ is minimal everywhere except possibly at 2 or 3. Thus it will be the case that $\lambda_{A,B}$ is 1, $2^{-12}, 3^{-12}$, or $6^{-12}$, depending respectively on whether the model $E_{A,B}$ is minimal everywhere, fails to be minimal only at 2, fails only at 3, or fails at both 2 and 3. Let $\operatorname{Cl}_{\lambda}$ be the set of residue classes mod $6^6$ such that if $(A,B) \in \Z^2$ reduces to a class in $\operatorname{Cl}_\lambda$, then $(A,B)$ are weakly minimal at 2 and 3, and $\lambda_{A,B}=\lambda$. The table below summarizes the values of $\# \operatorname{Cl}_\lambda$, which were calculated using Tate's algorithm (which can be found in \cite{BigSil}), supplemented by some calculations with \texttt{sage} \cite{sage}.

\vspace{.5cm}
\begin{tabular}{|l|l|l|}\hline
Model is... & Factor $\lambda$ & Size of $\operatorname{Cl}_\lambda$ \\\hline\hline
Minimal everywhere & 1 & $(2^{12}-12-2^2)(3^{12}-18-3^2)$\\
Minimal except at 2 & $2^{-12}$ & $12 \times(3^{12}-18-3^2)$\\
Minimal except at 3 & $3^{-12}$ & $(2^{12} - 12-2^2)\times 18$\\
Minimal except at 2 and 3 & $6^{-12}$ & $12 \times 18$\\\hline
\end{tabular}
\vspace{.5cm}

Fix two residue classes $A_0$ mod $6^6$ and $B_0$ mod $6^6$ so that this choice corresponds to a particular value $\lambda = \lambda_{A_0,B_0} = 1, 2^{-12}, 3^{-12}$ or $6^{-12}$. Then for weakly minimal $(A,B)$ corresponding to these classes, $H(E_{A,B})<X$ is equivalent to $(A,B) \in R_{X/\lambda}$. We want to calculate
\begin{equation*}
\#\{ (A,B) \in \Z^2 \mid (A,B) \equiv (A_0, B_0) \text{ mod } 6^6, (A,B) \text{ weakly minimal}, (A,B) \in R_{X/\lambda}\}
\end{equation*}
since summing these over all residue classes will give the number of elliptic curves of Faltings height less than $X$ up to isomorphism. We do this in the next section.

\section{Counting weakly minimal lattice points of a fixed residue class}

\begin{prop}\label{fixedresidueclass}
Fix two residue classes $A_0, B_0$ mod $6^6$ such that $(A_0, B_0)$ is weakly minimal with respect to $2$ and $3$. Then there is a constant $M$ such that the number of lattice points in $R_X$ reducing mod $6^6$ to $(A_0,B_0)$ such that $\Delta_{A,B} \not = 0$ is $\sigma 6^{-12} X^{5/6} + O(X_M^{1/2}(\log X_M)^3)$ where $X_M = \max(X,M)$ and $\sigma = \Area(R_1)$.
\end{prop}

\begin{proof}
The general idea in this proof is the classical one of approximating the number of lattice points by the area of a region. Since the error term is dependent on the length of the boundary, we will need to use a region that has finite boundary. (In the end, we will be doing this to the region scaled by $6^{-6}$ since counting all integral points in that rescaled region will be approximately the same as counting the points that satisfy the particular congruence. For simplicity, we ignore the rescaling for now.)

Let $X_M=\max(M,X)$. Then \[R_X' = \left\{ (A,B) \in R_X \; \bigg| \; |A|\le NX_M^{1/3}(\log X_M)^2\right\}\] where $\N, \X$ are from Lemma~\ref{heightbig}. If $X\ge\X$, then Lemmas \ref{discsmall} and \ref{heightbig} tells us that the lattice points $(A,B) \in R_X \setminus R_X'$ satisfy $\Delta_{A,B}=0$. If $X<\X$, then $R_X \subseteq R_{\X}$, so $R_X \setminus R_X' \subseteq R_{\X}\setminus R_{\X}'$, and so also $\Delta_{A,B} = 0$.

Recall that $C$ is the bound on $|\Delta(\tau)\Im(\tau)^6|$ from Lemma~\ref{DiscBound}, and $c=~(27/4)^{1/3}$. Using Lemma~\ref{InTail}, there is a $\beta_0>0$ such that if $|A|\le N X_M^{1/3} (\log X_M)^2$ and $(A,B) \in R_X$, $|B|\le~\beta_0 N^{3/2} X_M^{1/2} (\log X_M)^3$. Thus, we can think of $R_X'$ as the intersection of $R_X$ with a rectangle. The boundary of $R_X'$ is contained in the union of the boundary of that rectangle (and thus no worse than $O(X_M^{1/2}(\log X_M)^3)$) and the boundary of $R_X$ (that is, the boundary of the closure of $R_X$). So we need only show that the curve given by the boundary of $R_X$ inside that rectangle has length less than $O(X_M^{1/2}(\log X_M)^3)$. The lemmas that follow establish this.

\begin{lemma}\label{Smallbox}
For any positive $\alpha, \beta \in \R$, the boundary of $R_X$ in the rectangle $|A|\le \alpha X^{1/3}$ and $|B|\le \beta X^{1/2}$ is $O(X^{1/2})$.
\end{lemma}

\begin{proof}
Since $R_X$ is $R_1$ scaled by $X^{1/3}$ in the $A$-axis and $X^{1/2}$ in the $B$-axis, it suffices to show that the boundary of $R_1$ in any rectangle is bounded. Since the scaling cannot make the boundary longer than the scaling itself, it will follow that the boundary in $R_X$ is $O(X^{1/2})$.

The boundary of $R_1$ is given by the zeros of \[|\Delta_{A,B}| - |\Delta(\tau_{A,B})| \Im(\tau_{A,B})^6\] This function is real analytic away from $\Delta_{A,B}=0$, and thus in any compact set, its zero set is rectifiable (see \cite{Fed} 3.4.10). \end{proof}


Note that \cite{Fed} 3.4.10 also implies that the boundary of $R_X$ in any rectangle is finite, but since we want to know how big our error is, we want to calculate it more carefully.

We must then consider the boundary in the region $\alpha X^{1/3}<|A|$ or $\beta X^{1/2} <|B|$ (noting that we can fix $\alpha, \beta$ as we like). We'll first consider $R_1$ and then generalize. For $(A,B)$ such that $\Delta_{A,B} \not = 0$, if we choose $\tau_{A,B}$ in the standard fundamental domain, then $q$ is real valued, and so is $\Delta(\tau_{A,B})$. In fact, it shares the same sign as $\Delta_{A,B}$. So
\begin{equation} F(A,B) = \Delta_{A,B} - \Delta(\tau_{A,B})\Im(\tau_{A,B})^6 \label{boundaryfunction}\end{equation}
is a real-valued function the zero set of which is the boundary of $R_1$. We will show that there is a $\beta$ such that for all $B>\beta$, each partial of $F$ on the boundary of $R_1$ has a constant sign, and the same hold for all $B<-\beta$. First, this implies that since $F$ is defined everywhere but where $\Delta_{A,B}=0$, there can be at most two components of the boundary (one for each connected region separated by the cubic). Second it tells us that where $|B|>\beta$, a connected curve given by the zero set of $F$ (and thus also the corresponding boundary of $R_X$) is such that its length in any rectangle is less than or equal to the length of the boundary of the rectangle. (This follows from the triangle inequality. This prevents by how much a curve can change directon and thus bounds its length.)

\begin{lemma}
There is a $\beta>0$ such that if $(A,B) \in \R^2$ such that $F(A,B)=0$ and $|B|>\beta$, then the gradient $\nabla F$ at $(A,B)$ is in the third quadrant if $B>0$ and in the second if $B<0$.
\end{lemma}

\begin{proof}
We calculate the partial derivatives of $F$ on the curve where $F=0$. The choice of $\tau_{A,B}$ implies that $q_{A,B} = e^{2\pi i \tau_{A,B}}$ is real valued, so we can write $F$ and $\Delta(\tau_{A,B})$ in terms of $q_{A,B}$ using $2\pi \Im(\tau_{A,B}) = \log |q_{A,B}|$ into \eqref{boundaryfunction}.

Note that Lemma~\ref{DenBound} bounds the second term in the definition of $F$ by $C$ (since we currently considering $R_1$), so $|\Delta_{A,B}|$ is bounded on $F(A,B) = 0$, and thus as $|B|$ grows large we have $4A^3 \approx 27B^2$. So $A^3$ grows large and thus also $j_{A,B}$, from which it follows that $q_{A,B} \rightarrow 0$.

Define \[J(\tau_{A,B}) = \frac{1}{j_{A,B}} =\frac{1}{1728} + \frac{B^2}{256A^3}.\] By the chain rule:
\[ \frac{\del F}{\del A} = \frac{\del \Delta_{A,B}}{\del A} - \frac{1}{(2\pi)^6}\left(\frac{d(\Delta(q)(\log|q|)^6)}{dq}\right) \left(\frac{dq}{dJ}\right) \left(\frac{\del J}{\del A}\right) \]
Since $q$ can be written as a convergent power series in $j^{-1} = J$ with leading term $J$, it follows that $\frac{dq}{dJ} = 1 + O(J) = 1+O(q)$. Similarly, since $\Delta(q)$ is a convergent power series in $q$ with leading term $q$, we also have $\frac{d\Delta(q)}{dq} = 1+O(q)$. Using that $q$ is real, we get that
\begin{align*}
\frac{d\Delta(q)(\log|q|)^6}{dq} &= (1+O(q))(\log |q|)^6 + 6 (1+O(q))(\log|q|)^5\\
& = (\log |q|)^6 + O((\log|q|)^5)
\end{align*}
Lastly, $\frac{\del \Delta_{A,B}}{\del A}$ and $\frac{\del J}{\del A}$ are straightforward calculations. Putting this together, we get that
\[\frac{\del F}{\del A} = -192A^2 - \frac{1}{(2\pi)^6} (\log |q_{A,B}|)^6 \left(-\frac{3B^2}{256 A^4}\right)(1+o(1))\]
and a similar calculation gives
\[\frac{\del F}{\del B} = -864B - \frac{1}{(2\pi)^6} (\log |q_{A,B}|)^6 \left(\frac{B}{128 A^3}\right)(1+o(1)).\]

We want to know the value of these partial derivatives where $F(A,B) = 0$, which means that \begin{equation} \Delta_{A,B} = \Delta(q_{A,B})(\log |q_{A,B}|)^6. \label{deltainq}\end{equation} Additionally, for small $q$, $\Delta(q) = q(1+o(1)) = 1/j (1+o(1))$; substituting this into \eqref{deltainq} yields \[(\log |q|)^6 = j \Delta_{A,B} (1+o(1)) = -1728(4A)^3(1+o(1)).\] Lastly, it follow from Lemma~\ref{InTail} (for $R_1$), that $B^2/A^3$ approaches $-4/27$. We put these two estimates into \eqref{boundaryfunction} to conclude that
\[\frac{\del F}{\del A} = -A^2 \left(192 - \frac{3}{\pi^6}(1+o(1))\right)\]
and
\[\frac{\del F}{\del B} =  -B \left(864 - \frac{27}{2 \pi^6} (1+o(1))\right)\qedhere\]
\end{proof}

By Lemma~\ref{InTail}, we can chose $\alpha$ such that if $(A,B)\in R_X$ and $|B|\le \beta X^{1/2}$ implies $|A|\le \alpha X^{1/3}$. Thus, for points where $|B|<\beta X^{1/2}$, Lemma~\ref{Smallbox} applies, and for those where $|B|\ge \beta X^{1/2}$, the above argument applies, and applying it with the two rectangles where $\beta X^{1/2}<|B|<\beta_0 N^{3/2} X^{1/2} (\log X)^3$ and $|A|< N X^{1/3} (\log X)^2$, the boundary of $R_X$ in that region is $O(X^{1/2}(\log X)^3)$.

Having shown that the boundary of $R_X'$ is $O(X_M^{1/2}(\log X_M)^3)$, we turn now to the lattice points of $R_X'$. If we wanted all lattice points, it is a standard result (for a proof, see \cite{Lan} Vol. 2, pg. 186) that if the boundary of a region is rectifiable, the difference between its area and the number of lattice points it contains is bounded by $4(L+1)$ where $L$ is the length of the boundary, thus $O(X_M^{1/2}(\log X_M)^3)$. The area of the region $R_X$ is $\sigma X^{5/6}$, and this differs from the area of $R_X'$ by the region where $|B|>\beta_0 N^{3/2} X^{1/2} (\log X)^3$. Lemma \ref{heightbig} tells us that for a fixed $B$ in that region, the width of the region is $O(B^{-4/3})$. So the area of the difference between $R_X$ and $R_X'$ (including both negative and positive $B$) must be bounded by \[ 2\int_{\beta_0 N^{3/2} X^{1/2} (\log X)^3}^\infty O(B^{-4/3}) dB = O(X^{-1/6} (\log X)^{-1})\] 
so we can conclude that the total number of lattice points in $R_X$ is $\sigma X^{5/6} + O(X_M^{1/2}(\log X_M)^3)$.

Lastly, we want to apply the entirety of the above argument after rescaling by $1/m$ in both directions. The number of lattice points of a particular residue class mod $m$ in a region is the number of all lattice points in the same region scaled by $1/m$ in both directions and translated appropriately. So scaling the region will scale area by $6^{-12}$ and length by $6^{-6}$. Translating will change the number of lattice points by at most the boundary, so the final result must be $\sigma 6^{-12} X^{5/6} + O(X_M^{1/2}(\log X_M)^3)$.

\end{proof}

\section{Conclusion of the proof of Theorem \ref{maintheorem}}

Recall that $R_X$ was defined so that\[ S_X = \{ \text{weakly minimal } (A,B) \in \Z^2 \mid \Delta_{A,B} \not = 0, (A,B) \in R_{X/\lambda_{A,B}}\}. \] Let \begin{align*}
S_{X, \lambda} &= \{ (A,B) \in S_X \mid \lambda_{A,B} = \lambda\}\\ 
&= \{ \text{weakly minimal } (A,B) \in \Z^2 \cap R_{X/\lambda} \mid \Delta_{A,B} \not = 0,\; \lambda_{A,B} = \lambda\}
\end{align*}
so that $S_X$ is a disjoint union of $S_{X,1}, S_{X, 2^{-12}}, S_{X, 3^{-12}},$ and $S_{X, 6^{-12}}$.

We now go about calculating the size of these. Recall from Section~\ref{resclasses} that $\operatorname{Cl}_{\lambda}$ is the set of residue classes mod $6^6$ such that if $(A,B)$ reduces to a class in $\operatorname{Cl}_{\lambda}$, then $\lambda_{A,B} = \lambda$ and $(A,B)$ is weakly minimal at 2 and 3. A M\"obius inversion argument shows that
\[ \sum_{\substack{d^4 \mid A\\ d^6 \mid B}} \mu(d) = \begin{cases}
1 & \text{ if }(A,B) \text{ is weakly minimal}\\
0 & \text{ otherwise},
\end{cases}\] so
\[ \# S_{X,\lambda} = \sum_{\substack{(A,B) \in \Z^2 \cap R_{X/\lambda}\\ (A,B) \text{ mod }6^6 \in \operatorname{Cl}_{\lambda}\\ \Delta_{A,B} \not = 0}} \;\sum_{\substack{d^4 \mid A\\ d^6 \mid B}} \mu(d).\]
Note that none of the points we are counting have $|A|> NX^{1/3} (\log X)^2$ (proved in Lemmas~\ref{discsmall} and \ref{heightbig}), so $d\le N^{1/4} (X/\lambda)^{1/12} (\log (X/\lambda))^{1/2}$. Additionally, since $(A,B)$ mod $6^6 \in \operatorname{Cl}_\lambda$, it's also true that $2\nmid d, 3\nmid d$.

Let $d^*\!\operatorname{Cl}_{\lambda}$ be the set of residue classes $(\bar{a},\bar{b}) \in (\Z/6^6\Z)^2$ such that $(d^4\bar{a},d^6\bar{b}) \in \operatorname{Cl}_\lambda$; note that $\# d^*\!\operatorname{Cl}_{\lambda} = \# \operatorname{Cl}_{\lambda}$. If $X>M$, there is a bijection between
\[ \{ (A,B,d) \in \Z^3 \mid (A,B) \in R_{X/\lambda}, (A,B) \text{ mod }6^6 \in \operatorname{Cl}_{\lambda}, \Delta_{A,B} \not = 0, d^4 \mid A, d^6 \mid B\}\] and \begin{multline*} \{(a,b,d) \in \Z^3 \mid d\le N^{1/4} (X/\lambda)^{1/12} (\log (X/\lambda))^{1/2}, 2\nmid d, 3 \nmid d, \\(a,b) \in R_{X/(\lambda d^{12})}, (a,b) \text{ mod }6^6 \in d^*\!\operatorname{Cl}_\lambda, \Delta_{a,b} \not = 0 \} \end{multline*}  given by $(A,B,d) \mapsto (Ad^{-4}, Bd^{-6}, d)$.

Then
\begin{align*}
\# S_{X,\lambda}& = \sum_{\substack{d<N^{1/4} (X/\lambda)^{1/12} (\log (X/\lambda))^{1/2} \\ 2\nmid d,\, 3 \nmid d}} \mu(d) \sum_{\substack{(a,b) \in \Z^2 \cap R_{X/(\lambda d^{12})}\\ (a,b) \text{ mod } 6^6 \in d^*\!\operatorname{Cl}_\lambda\\ \Delta_{a,b} \not = 0}} 1\\
&= \sum_{\substack{d< N^{1/4} (X/\lambda)^{1/12} (\log (X/\lambda))^{1/2} \\ 2 \nmid d,\, 3 \nmid d}} \mu(d) \bigg(\frac{\# \operatorname{Cl}_\lambda \sigma}{6^{12}} (X/\lambda d^{12})^{5/6} + \\ & \hspace{1.4in} O\left(\max(M, (X/d^{12})^{1/2}) (\log(\max(M, X/d^{12})))^3 \right) \bigg)
\end{align*}
by Proposition~\ref{fixedresidueclass}. It is simple to show that \begin{equation} \sum_{\substack{d<Y \\ 2\nmid d, 3 \nmid d}} \mu(d)d^{-10} = \frac{1}{\zeta(10)(1-2^{-10})(1-3^{-10})} + O(Y^{-9})\end{equation}\label{moebius} so 
\[ \# S_{X,\lambda} = \frac{\# \operatorname{Cl}_\lambda \sigma}{\lambda^{5/6} 6^{12} (1-2^{-10})(1-3^{-10})} \zeta(10)^{-1} X^{5/6}\] with an error of \[\sum_{\substack{d< N^{1/4} (X/\lambda)^{1/12} (\log (X/\lambda))^{1/2} \\ 2 \nmid d,\, 3 \nmid d}} \mu(d) \max(M, (X/d^{-12})^{1/2}) (\log \max(M, X/d^{12}))^3\] (which dominates the error coming from applying Equation~\ref{moebius}). Using basic algebra and applying Equation~\ref{moebius} again, it can be shown that this error is \[ O(X^{1/2} (\log X)^3).\]
Summing over the four possible values of $\lambda$ using the calculations of $\# \operatorname{Cl}_\lambda$ from Section~\ref{resclasses}, this gives us
\[ 12\sigma \zeta(10)^{-1} X^{5/6} + O(X^{1/2}(\log X)^3).\]

\section*{Acknowledgements}

I would like to thank my advisor Bjorn Poonen for suggesting this problem, for many helpful conversations, and good advice at all stages. I appreciate the help of William Minicozzi for pointing me to the reference \cite{Fed}, Henri Cohen for help calculating the integral for $\sigma$, and many others with whom I discussed this project. Thank you also to the referee for this paper, who gave many thoughtful and helpful suggestions. 

This research was supported in part by National Science Foundation grants DMS-1069236 and DMS-0943787.  Any opinions, findings, and conclusions or recommendations expressed in this material are those of the author(s) and do not necessarily reflect the views of the National Science Foundation.

\bibliographystyle{plain}
\bibliography{references}

\end{document}